\pdfoutput=1 %pdflatex
\documentclass[10pt,a4paper]{article}
\usepackage[utf8]{inputenc}
\usepackage{amsmath,amscd}
\usepackage{amsfonts}
\usepackage{amssymb}
\usepackage{amsthm}
\usepackage{xcolor}
\usepackage{enumitem}
\usepackage{tikz-cd}
\usepackage{pdfpages}
\usepackage{hyperref}

\DeclareMathOperator{\Pic}{Pic}

\DeclareMathOperator{\Aut}{Aut}

\newtheorem{theorem}{Theorem}

\newtheorem{definition}[theorem]{Definition}

\newtheorem{proposition}[theorem]{Proposition}
\newtheorem{remark}[theorem]{Remark}

\newtheorem{rem}[theorem]{Remark}

\def\m{\mathbb}

\date{}

\begin{document}
%\date{}
\title{Computation of Singular Godeaux Surfaces\\
and a New Explicit Fake Quadric}

\author{Carlos Rito}
\maketitle
\vspace{-.8cm}
\begin{center}
\large With an Appendix by Christian Gleissner and Noah Ruhland
\end{center}
\vspace{.2cm}

\begin{abstract}
We present a computational method for detecting highly singular members in families of algebraic varieties. Applying this approach to a family of numerical Godeaux surfaces, we obtain explicit examples with many singularities. In particular, we construct a Godeaux surface whose singular locus consists of two $\mathsf A_1$ and two $\mathsf A_3$ singularities. We show that this surface admits a $\mathbb{Z}/2 \times \mathbb{Z}/4$ abelian cover which is a smooth minimal surface of general type with invariants $K^2=8$ and $p_g=0$, i.e.\ a fake quadric. Together with the result in the Appendix, this provides the first explicit construction of a fake quadric that does not arise as a quotient of a product of curves.


\end{abstract}

\section{Introduction}

A \emph{fake quadric} is a smooth minimal complex algebraic surface of general type with the same invariants
as the quadric surface in $\mathbb P^3$, namely $K^2=8,$ \mbox{$p_g=q=0$}.
These surfaces form an important class in the geography of surfaces of general type. While fake projective planes are known to be uniformized by the complex $2$-ball, the situation for fake quadrics is much less clear. All examples constructed so far are uniformized by the bidisk $\mathbb{H}\times \mathbb{H}$, but determining whether there exists a fake quadric with a different universal cover remains a central open problem in the area.

As explained e.g. in \cite{DzRo},
the fake quadrics that are uniformized by the bidisk can be divided into two classes:
surfaces isogenous to a product of curves, which may admit moduli, and quaternionic fake quadrics, which are rigid.
The former have been classified by Bauer, Catanese and Grunewald \cite{BCG},
the latter are analogues of fake projective planes, sharing with them both rigidity and the challenge of explicit construction.
Indeed, the first explicit construction (by equations) of a fake projective plane was obtained only relatively recently
by Borisov and Keum \cite{BK}, highlighting the difficulty of such constructions.

The motivation for this work arose from the study of automorphisms of quaternionic fake quadrics by
D\v{z}ambi\'c and Roulleau \cite{DzRo}.
They showed that if such a surface admits a group of automorphisms isomorphic to $(\mathbb Z/2)^3$ or $\mathbb Z/2\times\mathbb Z/4$,
then its quotient is a numerical Godeaux surface with singular set $7\mathsf{A}_1$ or $2\mathsf{A}_1+2\mathsf{A}_3,$ respectively.
But they weren't able to produce examples for these cases.
Then results of Bauer-Pignatelli \cite{BP} and Frapporti-Pignatelli \cite{FP} imply that no $\mathbb Z/2$-Godeaux surface with these singularities can occur as a quotient of a product of curves.
This suggested that precisely these configurations might lead to interesting new examples of fake quadrics.

The present paper has two main contributions. First, we introduce an efficient computational method for detecting highly singular members in families of algebraic varieties, based on interpolation over finite fields and lifting to characteristic zero.
Second, applying this to a family of $\mathbb Z/2$-Godeaux surfaces, we obtain explicit surfaces with many singularities,
in particular a 2-parameter family of $\mathbb Z/2$-Godeaux surfaces with 6 nodes.
We were unable to identify a 7-nodal surface in this family (we conjecture its non-existence),
but we found a $\mathbb{Z}/2$-Godeaux surface with singular set \mbox{$2\mathsf{A}_1+2\mathsf{A}_3.$}
We show that it admits a $\mathbb{Z}/2 \times \mathbb{Z}/4$ abelian cover which is smooth, minimal,
and has invariants $K^2=8$, $p_g=0$, hence a fake quadric.

This surface is defined over the rational numbers, and hence invariant under complex conjugation.
Although the defining equations involve only relatively small rational coefficients,
the computation of the auxiliary curves providing the divisibility conditions required for the
abelian cover proved to be highly demanding. Over finite fields these computations were already difficult,
and lifting them to characteristic zero introduced enormous rational numbers.
In fact, it was necessary to repeat the calculations for more than 600 distinct primes in order to achieve a successful lifting.

A crucial input is a result established in the Appendix by Gleissner and Ruhland: every automorphism of a variety isogenous to a product of curves lifts to an automorphism of the covering curves. This ensures that our example does not arise from a product of curves. Combining their result with our construction, we obtain the
\emph{first explicit construction of a fake quadric that is not a quotient of a product of curves}.

Whether our surface is uniformized by the bidisk or not remains open. If it is, this would yield the first explicit construction of a quaternionic fake quadric. Unlike fake projective planes, which are completely classified and where the geometry of one example was sufficiently understood to guide the explicit construction, quaternionic fake quadrics are not classified.
The work of Linowitz, Stover and Voight \cite{LSV} provides only a list of hypothetical maximal lattices that
might contain the group of a fake quadric, but offers no indication of the geometry of such a surface.

All computations were performed in Magma \cite{BCP} and are available as arXiv ancillary files.
The file Main.txt loads the other scripts containing the large defining equations and performs
all necessary verifications; to keep the presentation manageable, it contains only the minimal
Magma verifications that certify our claims, omitting the demanding computations that led to them.

\subsubsection*{Notation}

As usual, the holomorphic Euler characteristic of a surface $S$ is denoted by $\chi(S),$ the geometric genus by $p_g(S),$
the irregularity by $q(S),$ and a canonical divisor by $K_S.$
A $(-m)$-curve is a curve isomorphic to $\m P^1$ with self-intersection $-m$.
Linear equivalence of divisors is denoted by $\equiv$.

\subsubsection*{Acknowledgments}

We would like to thank Margarida Mendes Lopes, Rita Pardini, Christian Gleissner, Roberto Pignatelli,
Matthew Stover and John Voight for interesting and useful conversations.

This research was partially financed by Portuguese Funds through FCT
(Funda\c c\~ao para a Ci\^encia e a Tecnologia) within the Project UID/00013:
Centro de Matem\'atica da Universidade do Minho (CMAT/UM).

%\section{Finding highly singular surfaces}

\section{A method to detect highly singular members in a family}\label{sec:method}

Let $\mathcal{F}$ be a family of algebraic varieties defined by explicit equations depending on $n$ parameters. 
Assume that the general member of $\mathcal{F}$ is smooth.
Our goal is to detect members of $\mathcal{F}$ with a large number of ordinary double points (nodes). 
The method we employ is based on the following inductive strategy.

\medskip

\noindent\textbf{Step 1: Detecting nodal members.}
The locus of varieties in $\mathcal{F}$ with at least one node is a subvariety $F_1$ of codimension $1$ in the parameter space. 
This locus is cut out by a polynomial equation
\[
   f_1 = 0
\]
in the parameter space. Our first task is to determine such a defining equation.

\medskip

\noindent\textbf{Step 2: Working over finite fields.}
To make the computation feasible, we pass to a finite field $\mathbb{F}_q$. 
Provided that the family $\mathcal{F}$ is sufficiently large, we expect to find many members over $\mathbb{F}_q$ with (at least) one node, 
without the need to extend the base field. 
Each such member corresponds to a rational point in the parameter space lying on the hypersurface $F_1$.

\medskip

\noindent\textbf{Step 3: Interpolation of $f_1$.}
By collecting sufficiently many points in the parameter space corresponding to one-nodal varieties, 
we recover the polynomial $f_1$ using interpolation. 
In practice, we employ the \texttt{LinearSystem} function in Magma to determine polynomials passing
through the given set of points. 
If $\deg(f_1)$ is not too large, this interpolation succeeds and yields the desired equation.

\medskip

\noindent\textbf{Step 4: Iteration for higher nodes.}
Once $f_1$ is known, we restrict to the subfamily $F_1$ defined by $f_1=0$ and repeat the process. 
We then search within this locus for members with two nodes, collect the corresponding parameter points, 
and interpolate to recover further defining equations. 
The resulting locus is a subfamily $F_2$ of codimension $2$ in the parameter space
(in general defined by more than two equations). 
This procedure need not be carried out one node at a time: at each step, 
we may impose the existence of several nodes simultaneously. 
Proceeding inductively, we obtain a nested sequence of subfamilies
\[
   F_1, F_2,\ldots
\]
whose final locus parametrizes varieties with many nodes.

\medskip

\noindent\textbf{Step 5: Field extensions.}
As the iteration proceeds and the dimension of the parameter subfamily decreases, it becomes increasingly difficult to find rational points over the base finite field. In such cases, it may be necessary to pass to field extensions in order to locate enough points. Success at this stage depends strongly on the particular situation.

\medskip

\noindent\textbf{Step 6: Lifting to characteristic zero.}
The defining equations obtained above are computed modulo primes, possibly in finite field extensions. 
We lift these equations to characteristic zero. 
Concretely, the coefficients of the interpolated equations are first collected modulo several good primes. 
If an extension field was required, the coefficients are expressed in terms of a fixed basis of the extension, 
so that each coefficient can be represented by a tuple of elements in the base finite field. 
These tuples are then lifted simultaneously via the Chinese remainder theorem (CRT), and \texttt{RationalReconstruction} in Magma
is applied to recover the corresponding rational numbers.

\medskip

\noindent
Although we have described the procedure for nodes, the method can be adapted without difficulty 
to search for varieties with other types of singularities.

\begin{remark}
The interpolation step relies on the efficient computation of linear systems of polynomials. 
In practice we use the \texttt{LinearSystem} functionality in Magma, which has recently been 
completely rewritten by the author. The new version is substantially faster and allows 
interpolation through very large sets of points (on the order of $30000$ points or more), 
significantly increasing the chances of recovering the correct defining equations.
\end{remark}

\section{A 4-dimensional family of $\,\mathbb{Z}/2$-Godeaux surfaces with four nodes}

In \cite{DR} it was shown that the moduli space of $\mathbb{Z}/2$-Godeaux surfaces is irreducible of dimension $8$.
The computations there were organized so that the output included a main $8$-dimensional family,
together with certain families of dimension $\leq 7$ lying in its closure.
In this section we focus on one such subfamily, denoted by $\mathcal M_2^1$.

\medskip

The existence of a node imposes one independent condition on the parameters.  
One would therefore expect the locus of surfaces in $\mathcal M_2^1$ with four nodes to have codimension~$4$.  
Surprisingly, computations revealed the existence of a subfamily of codimension~$3$, i.e.\ of dimension~$4$, parametrizing surfaces with four nodes.  
This was detected and confirmed via the interpolation method described in Section~\ref{sec:method}.

\medskip

\noindent\textbf{Step 1: Producing nodal Godeaux surfaces efficiently.}  
The first task is to generate many $\mathbb{Z}/2$-Godeaux surfaces with nodes, to provide enough data for interpolation.  
This is made possible by working with the universal cover: recall that the universal cover of a $\mathbb{Z}/2$-Godeaux surface is a surface with invariants $p_g=1$, $q=0$, $K^2=2$. Its bicanonical map realizes it as an octic surface in $\mathbb{P}^3$, given by an explicit homogeneous equation
\[
   f_8(x_0,\ldots,x_3; p_1,\ldots,p_7) = 0,
\]
where the $p_i$ are the parameters.  

To detect singular members rapidly, we regard $f_8$ in the extended space with coordinates
\[
x_0,\ldots,x_3,\quad p_1,\ldots,p_7
\]
and impose the vanishing of the partial derivatives of $f_8$ with respect to $x_0,\ldots,x_3$. 
This gives a system of equations that detects when a given point \mbox{$(x_0:x_1:x_2:x_3)$}
is a singular point of the surface 
for the given choice of parameters $(p_1,\ldots,p_7)$.  

By sampling the parameters randomly, one can efficiently find many surfaces with at least one node. 
This procedure is fast because the condition of having a single node is codimension one in the parameter space, 
so such surfaces appear with high probability in random sampling. 
Once one node is obtained, further nodes can be produced by continuing the random search. 
Since the locus of surfaces with four nodes has dimension $4$, this random method is sufficient to find 
a large supply of examples for our purposes.

\medskip

\noindent\textbf{Step 2: Detecting the locus of four-nodal surfaces.}  
A first indication of this unexpected codimension came from an abundance of rational points
in the parameter space of $\mathcal M_2^1$ corresponding to four-nodal surfaces. 
The density of such points suggested the presence of a component of codimension $3$ rather than $4$. 

\medskip

\noindent\textbf{Step 3: Interpolation.}  
As explained in Section~\ref{sec:method}, we carried out interpolation using the \texttt{LinearSystem} functionality in Magma.  
To avoid complications arising from the possible presence of multiple components, the interpolation was performed with 
linear systems $L$ of polynomials through $n$ points, with $n$ only slightly larger than $\dim L$. 
This ensures that the resulting polynomials capture all relations satisfied by the chosen points, but avoids overfitting 
to isolated solutions.

The interpolation produced a collection of polynomials defining a subvariety $Z \subset \mathcal M_2^1$. 
At this stage it was unclear whether $Z$ was irreducible,
or whether it consisted of a genuine component together with a number of isolated points.

\medskip

\noindent\textbf{Step 4: Removing isolated points.}  
To separate the desired component, we computed the dimension of the tangent space of $Z$ at 
each point $p$ used in the interpolation.
Whenever this dimension was zero, the point $p$ was discarded, since it could not lie on a positive-dimensional component.  

After removing all such points, we repeated the interpolation using the remaining set.  
The resulting equations defined a 4-dimensional subvariety of $\mathcal M_2^1$, free from the isolated solutions.

\section{A $\mathbb{Z}/2$-Godeaux surface with singular set\\ $2\mathsf A_1+2\mathsf A_3$}

Continuing with the method described above, we detected a $2$-dimensional family of $\mathbb{Z}/2$-Godeaux surfaces 
with six nodes. Recall that an $\mathsf A_1$ singularity is an ordinary double point, whose minimal resolution introduces a single $(-2)$-curve, 
while an $\mathsf A_3$ singularity is a rational double point whose resolution produces a chain of three $(-2)$-curves. 
The degeneration of two nodes into an $\mathsf{A}_3$ singularity imposes one further condition on the parameter space.

\medskip

\noindent\textbf{Step 1: From nodes to higher singularities.}  
Within the $2$-dimensional family of six-nodal surfaces, we searched for members with more complicated singularities. 
By repeating the interpolation procedure on this family, we obtained a $1$-dimensional locus of surfaces with singular set $4\mathsf A_1 + 1\mathsf A_3$. 

\medskip

\noindent\textbf{Step 2: Specialization to $2\mathsf A_1+2\mathsf A_3$.}  
Over several finite fields $\mathbb{F}_p$, we examined the resulting families point by point. 
In each case we were able to locate two surfaces whose singular sets are of type $2\mathsf A_1+2\mathsf A_3$,
thus providing candidates for the target configuration.

\medskip

\noindent\textbf{Step 3: Lifting to characteristic zero.}  
For each prime $p_i$ in our computations, the equations admitted two possible solutions, $a_i$ and $b_i$, corresponding to the same unknown rational number.  
Since we cannot choose between $a_i$ and $b_i$ directly, we instead work with the data
\[
   (x-a_i)(x-b_i) = x^2 - (a_i+b_i)x + a_i b_i,
\]
whose coefficients are independent of the ordering of the two roots.  
As above, we lift these coefficients from the finite fields via the Chinese remainder theorem and
\texttt{RationalReconstruction} in Magma.  

\medskip

\noindent
This process produces a sequence of integer numbers and a sequence of rational numbers. By choosing the integer ones
we show the existence of a $\mathbb{Z}/2$-Godeaux surface with singular set $2\mathsf A_1+2\mathsf A_3$
inside the previously constructed family of six-nodal surfaces
(the rational ones give a non-isomorphic Godeaux surface with the same type of singular set, but we don't study it here).

\section{Searching divisibility relations on $X'$}\label{divisibility}

Let $X$ be the $\mathbb{Z}/2$-Godeaux surface with singular set $2\mathsf A_1+2\mathsf A_3$ constructed above, 
and let $X'$ be its smooth minimal resolution. 
The resolution introduces a total of eight exceptional $(-2)$-curves: 
$N_1, N_2$ from the two nodes, $N_3,N_4,N_5$ from the first $\mathsf A_3$, 
$N_6,N_7,N_8$ from the second $\mathsf A_3$.
The intersection matrix of the curves $N_1,\ldots,N_8$ and $K_{X'}$ is:
\begin{verbatim}
[-2  0  0  0  0  0  0  0  0]
[ 0 -2  0  0  0  0  0  0  0]
[ 0  0 -2  1  0  0  0  0  0]
[ 0  0  1 -2  1  0  0  0  0]
[ 0  0  0  1 -2  0  0  0  0]
[ 0  0  0  0  0 -2  1  0  0]
[ 0  0  0  0  0  1 -2  1  0]
[ 0  0  0  0  0  0  1 -2  0]
[ 0  0  0  0  0  0  0  0  1]
\end{verbatim}

Since
\[
   b_2(X') \;=\; 12\,\chi(\mathcal{O}_{X'}) - K_{X'}^2 + 4q(X') - 2 \;=\; 9,
\]
there are no other numerically independent curves on $X'$.

\medskip

In order to prove the existence of a $\mathbb{Z}/2 \times \mathbb{Z}/4$-abelian cover ramified over the singularities, 
we must find divisors on $X'$ that give the required $2$- and $4$-divisibility relations. 
We attempt to determine the “shape” of possible divisors by testing for numerical dependencies. 
The idea is to append to the list $\{N_1,\dots,N_8,K_{X'}\}$ a hypothetical curve $C$, 
write down the full intersection matrix of these $10$ curves, 
and then vary the intersection numbers of $C$ with the existing ones randomly. 
For each trial, we compute the nullspace of the resulting intersection matrix.  

The most promising relations obtained in this way are the following:
\begin{equation}\label{eq:divisor-relations}
\begin{aligned}
8K_{X'} &\equiv 4C' + 2N_1 + N_3 + 2N_4 + 3N_5 + 3N_6 + 2N_7 + N_8, \\
4K_{X'} &\equiv 2D' + N_1 + N_2 + N_6 + 2N_7 + N_8.
\end{aligned}
\end{equation}
Here $C',$ $D'$ denote the strict transforms on $X'$ of curves $C,$ $D$ on $X$ that will be computed in the next section.

\section{Computation of the curves $C$ and $D$}\label{eqsCD}

Let $X$ be the $\mathbb{Z}/2$-Godeaux surface with singular set $2\mathsf A_1+2\mathsf A_3$, and let $C,D \subset X$ denote the images 
of the hypothetical curves $C',D'$ on the minimal resolution $X'$ that appear in the divisor relations \eqref{eq:divisor-relations}.  
In this section we describe how explicit equations for $C$ and $D$ were obtained,
and we use them to conclude that (\ref{eq:divisor-relations}) indeed hold.

\medskip

We start from explicit equations for the \'etale $\mathbb{Z}/2$-cover $Y$ of $X$, 
realized as a singular surface in $\mathbb{P}^7$.  
From there we map the surface $X$ into $\mathbb{P}^6$ via 
its 4-canonical map, obtaining a surface containing 
the four isolated singularities of type
$2\mathsf A_1 + 2\mathsf A_3$ (and an additional disjoint singular curve).   

\medskip

\noindent\textbf{Computation of $D$.}  
To find $D \subset X$, we consider the linear system of hyperplanes in $\mathbb{P}^6$ 
passing through three of the singular points: the two nodes and one $\mathsf A_3$.  
We search for hyperplanes whose intersection with $X$ has multiplicity $2$ along a curve.  
Working over several finite fields $\mathbb{F}_{p}$, random sampling produces, for many $p$,
two distinct such hyperplanes. Their product defines a quadric containing $D$.  
This quadric is lifted to characteristic zero via the Chinese remainder theorem and 
\texttt{RationalReconstruction} in Magma.

\medskip

\noindent\textbf{Computation of $C$.}  
To find $C \subset X$, we consider the linear system of hyperplanes in $\mathbb{P}^6$ 
passing through the two $\mathsf A_3$ points and one node.  
We search for hyperplanes whose intersection with $X$ splits into two irreducible components of the same degree.  
As above we lift the result to characteristic zero, and we take $C$ to be one of these components.
Then we use the Magma function \texttt{LinearSystem} to find an element of $|8K_X|$ that cuts the surface $X$
with multiplicity 4 at the curve $C$
(as before, first over finite fields then lifting).

%\medskip
%\noindent\textbf{Divisibilities.}
%Let $C', D' \subset X'$ denote the strict transforms of $C, D \subset X.$
\begin{proposition}\label{relations}
The relations \eqref{eq:divisor-relations} hold as linear equivalences in $\Pic(X')$.
\end{proposition}

\begin{proof}
There exist positive integers $a_i, b_i$ such that
\begin{equation*}
\begin{aligned}
8K_{X'} &\equiv 4C' + a_1N_1 + a_3N_3 + a_4N_4 + a_5N_5 + a_6N_6 + a_7N_7 + a_8N_8, \\
4K_{X'} &\equiv 2D' + b_1N_1 + b_2N_2 + b_6N_6 + b_7N_7 + b_8N_8.
\end{aligned}
\end{equation*}

We check computationally that both $C$ and $D$ are irreducible and smooth.
Thus $C'N_i\leq 1$ and $D'N_i\leq 1,$\ $\forall i$.

To find the numerical divisibility relations involving the curves $N_1,\ldots,N_8,$ $C'$ and $D',$
we proceed as explained in Section \ref{divisibility} and compute the nullspace of the intersection
matrix for all possible cases,
using the information we have about those curves from construction.
This is available in one of the arXiv ancillary files.

These computations show that the numbers $a_i, b_i$
are exactly as predicted by equations \eqref{eq:divisor-relations},
thereby confirming that these relations do indeed hold.
\end{proof}

\section{The $\mathbb{Z}/2 \times \mathbb{Z}/4$ abelian cover}

We now construct the abelian cover used in the sequel.  We recall first the part of
Pardini's theory that is needed for this construction.

Let $Y$ be a smooth surface and let $G$ be a finite abelian group.  A normal abelian
$G$-cover $\pi\colon V\to Y$ is described by its \emph{building data}
\cite[Theorem~2.1]{Pa}.  These data consist of
\begin{itemize}[leftmargin=*]
\item for every cyclic subgroup $H$ of $G$ and every generator
$\psi$ of the group of characters $H^*$, a reduced effective divisor $D_{H,\psi}$ on $Y$, with total
support having normal crossings;
\item for every character $\chi\in G^*$, a divisor class $L_\chi\in\Pic(Y)$.
\end{itemize}
If $\chi|_H=\psi^{a_\chi}$ with $0\leq a_\chi<|H|$, then the building data must
satisfy the relations
\begin{equation}\label{eq:pardini-relations}
 L_\chi+L_{\chi'}\equiv L_{\chi\chi'}+
 \sum_{H,\psi}\epsilon^{H,\psi}_{\chi,\chi'}D_{H,\psi},
\end{equation}
where
\[
 \epsilon^{H,\psi}_{\chi,\chi'}=
 \begin{cases}
    0, & \text{if } a_\chi+a_{\chi'} < |H| \\
    1, & \text{otherwise. } 
\end{cases}
\]
Conversely, if these relations hold, Pardini's construction produces a normal
$G$-cover with the prescribed inertia groups along the components of the branch
divisor.

We have the following equivalent \emph{reduced form} \cite[Proposition~2.1]{Pa}.  Suppose that
$G^*$ is the direct sum of the cyclic subgroups generated by characters $\chi_1,\ldots,\chi_s$ of orders
$d_1,\ldots,d_s$.  For a fixed pair $(H,\psi)$ write
$\chi_i|_H=\psi^{a_i(H,\psi)}$, with $0\leq a_i(H,\psi)<|H|$.  Then it is enough
to prescribe the divisors $D_{H,\psi}$ and divisor classes $L_i=L_{\chi_i}$ satisfying
\begin{equation}\label{eq:reduced-pardini-relations}
 d_iL_i\equiv
 \sum_{H,\psi}\frac{d_i\,a_i(H,\psi)}{|H|}\,D_{H,\psi},
 \qquad i=1,\ldots,s.
\end{equation}
The remaining $L_\chi$ are then obtained from \eqref{eq:pardini-relations}.

\begin{proposition}\label{prop:existence-abelian-cover}
Let $X'$ be the smooth $\mathbb Z/2$-Godeaux surface
constructed above, with exceptional curves $N_1,\ldots,N_8$ labelled as in
Section~\ref{divisibility}.  There exists a normal
$G$-cover
\[
        \pi'\colon S'\longrightarrow X',
        \qquad G\cong \mathbb Z/2\times \mathbb Z/4,
\]
whose reduced branch data are supported on $N_1,\ldots,N_8$.
\end{proposition}

\begin{proof}
Let $G=\mathbb Z/2\times\mathbb Z/4$.  We choose generators
$\chi_2,\chi_5$ of $G^*$ of orders $2$ and $4$, respectively; in the Magma
notation used below these are the characters $T[2]$ and $T[5]$.

The group $G$ has three cyclic subgroups of order $2$ and two cyclic subgroups of
order $4$.  Denote them by
\[
  H_1,H_2,H_3,H_4,H_5,
  \qquad |H_1|=|H_2|=|H_3|=2,
  \quad |H_4|=|H_5|=4.
\]
For each $H_j$ choose a generator $\psi_j$ of $H_j^*$.  The reduced branch
divisors are
\begin{equation}\label{eq:chosen-branch-divisors}
\begin{array}{c|c}
(H,\psi) & D_{H,\psi} \\\hline
(H_1,\psi_1) & N_4+N_7 \\
(H_2,\psi_2) & N_1 \\
(H_3,\psi_3) & N_2 \\
(H_4,\psi_4) & N_3 \\
(H_4,\psi_4^{-1}) & N_5 \\
(H_5,\psi_5) & N_6 \\
(H_5,\psi_5^{-1}) & N_8 .
\end{array}
\end{equation}
The union of these curves is a normal-crossing divisor: the only intersections are
those in the two $\mathsf A_3$-chains, and all intersections are transverse.

For each character $\chi$ and each pair $(H,\psi)$ there is a unique integer
$0\leq a_\chi(H,\psi)<|H|$ such that
$\chi|_H=\psi^{a_\chi(H,\psi)}$.  The following Magma code computes these
integers.
\begin{verbatim}
function ExponentFromCharRestriction(Chi,H,psi)
  res:=Restriction(Chi,H);
  for n in [0..Order(H)-1] do
    if res eq psi^n then return n;end if;
  end for;
end function;

G:=SmallGroup(8,2);
T:=CharacterTable(G);
C:=[q`subgroup:q in CyclicSubgroups(G)|Order(q`subgroup) ne 1];
S:=[*[q:q in Characters(H)|Order(q) eq Order(H)]:H in C*];
[[[Order(Chi)*ExponentFromCharRestriction(Chi,C[i],psi)/Order(C[i])
   :psi in S[i]]:i in [1..#S]]:Chi in [T[2],T[5]]];
\end{verbatim}
The output is
\begin{verbatim}
[
[[0],[1],[1],[0,0],[1,1]],
[[2],[2],[0],[1,3],[3,1]]
]
\end{verbatim}
Notice that here $T[2]$ and $T[5]$ are characters such that $G^* \cong \mathbb{Z}/2 \times \mathbb{Z}/4$ is the direct sum of the subgroups they generate.

Summing up, for the ordered list of branch divisors in
\eqref{eq:chosen-branch-divisors}, the coefficients appearing in
\eqref{eq:reduced-pardini-relations} are
\begin{equation}\label{eq:coefficient-table}
\begin{array}{c|ccccccc}
       & N_4+N_7 & N_1 & N_2 & N_3 & N_5 & N_6 & N_8 \\\hline
\chi_2 & 0&1&1&0&0&1&1 \\
\chi_5 & 2&2&0&1&3&3&1 .
\end{array}
\end{equation}
Thus Pardini's reduced relations become
\begin{equation}\label{eq:reduced_data}
\begin{aligned}
2L_2 &\equiv N_1+N_2+N_6+N_8, \\
4L_5 &\equiv 2N_1+N_3+2N_4+3N_5+3N_6+2N_7+N_8.
\end{aligned}
\end{equation}
Comparing with relations \eqref{eq:divisor-relations} (see Proposition \ref{relations}), we finally take
\begin{equation}
\begin{aligned}
L_2 &:\equiv 2K_{X'}-D'-N_7, \\
L_5 &:\equiv 2K_{X'}-C'
\end{aligned}
\end{equation}
and the $G$-cover of $X'$ is well defined.
\end{proof}

\section{The fake quadric}

\begin{proposition}\label{prop:fake-quadric}
Let $S$ be the surface obtained from the covering in Proposition~\ref{prop:existence-abelian-cover}
by contracting the curves lying over the exceptional divisor of $X'\to X$.  Then
$S$ is a smooth minimal surface of general type with
\[
       K_S^2=8,\qquad p_g(S)=q(S)=0.
\]
\end{proposition}

\begin{proof}
We have a commutative diagram
\[
\begin{CD}
S' @>>> S \\
@VVV     @VVV \\
X' @>>> X
\end{CD}
\]

The $\mathbb{Z}/2 \times \mathbb{Z}/4$ covering $S'\longrightarrow X'$ splits as three double coverings,
each ramified on 4 disjoint $(-2)$-curves, thus $S$ is smooth.

The canonical divisor $K_S$ is equivalent to the pullback of $K_X,$ hence $K_S^2=8.$

In order to compute the geometric genus $p_g(S)$ and the holomorphic Euler chracteristic $\chi(S),$ 
we need to describe the remaining divisors 
$$L_3,\ L_4,\ L_6,\ L_7,\ L_8,\ \ \ \ L_i:=L_{\chi_i},$$ that appear in the building data of the cover.
We give a Magma function that computes each number $\epsilon^{H,\psi}_{\chi,\chi'}$ that appears in (\ref{eq:pardini-relations}).
Then it is easy to obtain all relations that give the building data of the cover.
The complete computer code is included in the ancillary files on the arXiv version of this paper.

The divisors are as follows:
\begin{equation*}
\begin{aligned}
L_6 &\equiv L_2+L_5-N_1-N_6, \\
L_7 &\equiv L_3+L_5-N_5-N_6, \\
L_8 &\equiv L_3+L_6-N_5-N_8, \\
L_3 &\equiv L_5+L_5-N_1-N_4-N_5-N_6-N_7, \\
L_4 &\equiv L_7+L_8-N_3-N_4-N_6-N_7-N_8.
\end{aligned}
\end{equation*}

From here we get that the full set of divisors is:
\begin{equation*}
\begin{aligned}
L_2 &\equiv 2K_{X'} - D' - N_7, \\
L_3 &\equiv 4K_{X'} - 2C' - N_1 - N_4 - N_5 - N_6 - N_7, \\
L_4 &\equiv 14K_{X'} - 6C' - D' - 3N_1 - N_3 - 3N_4 - 4N_5 - 5N_6 - 4N_7 - 2N_8, \\
L_5 &\equiv 2K_{X'} - C', \\
L_6 &\equiv 4K_{X'} - C' - D' - N_1 - N_6 - N_7, \\
L_7 &\equiv 6K_{X'} - 3C' - N_1 - N_4 - 2N_5 - 2N_6 - N_7, \\
L_8 &\equiv 8K_{X'} - 3C' - D' - 2N_1 - N_4 - 2N_5 - 2N_6 - 2N_7 - N_8.
\end{aligned}
\end{equation*}

Finally, using relations \eqref{eq:divisor-relations} we obtain
$$L_4 \equiv 6K_{X'} - 2C' - D' - N_1 - N_4 - N_5 - 2N_6 - 2N_7 - N_8.$$

We show computationally (see the ancillary arXiv files) that the linear system $|7K_{X} - 2C - D|$ is empty,
thus $|7K_{X'} - 2C' - D'|$ is also empty, which implies that
$h^0(X',K_{X'}+L_4)=0.$ Analogously we show that
$$h^0(X',K_{X'}+L_i)=0,\ \ \ i=2,\dots,8.$$

Now from  \cite[Proposition 4.1]{Pa}, we compute the geometric genus of $S$:
$$p_g(S)=p_g(X')+\sum_{i=2}^8 h^0(X',K_{X'}+L_i)=0.$$

The holomorphic Euler characteristic follows from  \cite[Proposition 4.2]{Pa}:
$$\chi(S)=8\chi(X')+\sum_{i=2}^8 \frac{1}{2}L_i(K_{X'}+L_i)=1.$$
\end{proof}

\section{$S$ is not a quotient of a product of curves}\label{sec:not-product-quotient}

\begin{theorem}\label{thm:not-product-quotient}
The fake quadric $S$ constructed above does not arise as a quotient of a product of curves.
\end{theorem}

\begin{proof}
First note that, if a fake quadric is a product-quotient surface, then it is
actually isogenous to a product.  Indeed, the product-quotient formula
\cite[Corollary~1.6]{BP}
\[
        K^2=8\chi-\frac{1}{3}B
\]
forces the basket correction term $B$ to vanish, because here $K_S^2=8$ and
$\chi(\mathcal O_S)=1$.  Thus the quotient model is smooth, and the action on
the product is free.  It is therefore enough to exclude the possibility that $S$ is
isogenous to a product.

Assume, for a contradiction, that
\[
        S\cong (C_1\times C_2)/\Gamma,
\]
where $C_1$ and $C_2$ are smooth curves of genus at least $2$ and $\Gamma$ acts
freely.  The surface $S$ carries the action of
$G\cong\mathbb Z/2\times\mathbb Z/4$ coming from the abelian cover construction.
By Theorem~\ref{thmAppendix} in the Appendix, every automorphism of a variety
isogenous to a product lifts to an automorphism of the covering product.  Hence
the $G$-action on $S$ lifts to an action on $C_1\times C_2$.

Thus there is a group $\widetilde\Gamma\subset\Aut(C_1\times C_2)$ such that
\[
        (C_1\times C_2)/\widetilde\Gamma
        \cong S/G=X.
\]
Therefore $X$ is a product-quotient surface.  The
action of $\widetilde\Gamma$ is either unmixed, in which case the results of
Bauer--Pignatelli \cite{BP} apply, or mixed, in which case the results of
Frapporti--Pignatelli \cite{FP} apply.

In the unmixed case, the classification of Bauer--Pignatelli shows that no product-quotient Godeaux surface has quotient basket of type
$
2\mathsf A_1+2\mathsf A_3 .
$
In the mixed case, the classification of Frapporti--Pignatelli allows this basket only for Godeaux surfaces whose fundamental group is $\mathbb Z/4$. However, the surface $X$ constructed above has fundamental group $\mathbb Z/2$. This contradiction shows that $S$ cannot be isogenous to a product, hence $S$ does not arise as a quotient of a product of curves.

\end{proof}

\bibliography{References}

\vspace{0.8cm}

\noindent Carlos Rito
\vspace{0.1cm}
\\ Centro de Matem\'atica, Universidade do Minho - Polo CMAT-UTAD
\vspace{0.1cm}
\\ Universidade de Tr\'as-os-Montes e Alto Douro, UTAD
\\ Quinta de Prados
\\ 5000-801 Vila Real, Portugal
\vspace{0.1cm}
\\ www.utad.pt, {\tt crito@utad.pt}

\newpage

\appendix
\begin{center}
  {\Large Appendix:\\ \vspace{.4em} The automorphism group of a variety\\ isogenous to a product \par}
  \vspace{1em}
  {\scshape Christian Gleissner and Noah Ruhland \par}
  \vspace{1em}
\end{center}

In this appendix, we determine the automorphism group of a variety isogenous to a product in arbitrary dimension, generalizing the results from the surface case, cf. \cite{Cat00}.

 \begin{definition}\label{VarIso}
 A complex variety $X$ is said to be isogenous to a product if it is isomorphic to a quotient 
 \[
X\simeq  (C_1 \times \ldots \times C_n)/G,
 \]
 where the $C_i$ are smooth projective curves of higher genus  and $G$ is a finite group acting freely on the product $C_1\times \ldots \times C_n$. 
 \end{definition}
 
 The key to determine  $\Aut(X)$ is to show that any automorphism of $X$ lifts to an automorphism of the product of curves. 
For this purpose we have to investigate the structure of the Galois group $G$ and identify the technical conditions that enables  us to lift. 
Note that any automorphism $g$ of a product $C_1 \times \ldots \times C_n$ of higher genus curves lifts to the universal cover which is a product of unit discs $\Delta^n$. By \cite[Proposition 3, p.68]{N71}, we have $$\Aut(\Delta^n)\simeq \Aut(\Delta)^n \rtimes \mathfrak S_n.$$ This  implies that $g$ is  diagonal, up to a  permutation of factors. Thus, for any group $ G\subset \Aut(C_1\times \ldots \times C_n)$ the diagonal subgroup 
 \[
 G^0:=G\cap [\Aut(C_1) \times \ldots \times \Aut(C_n)]
 \]
 is normal and the quotient group $G/G^0$ embeds into $\mathfrak S_n$. The diagonal subgoup $G^0$ is also given as the intersection of the groups 
 \[
G_i := G\cap [\Aut(C_1 \times \ldots \times \widehat{C}_i \times \ldots \times C_n) \times \Aut(C_i)].
\]
The group $G_i$  is the largest subgroup of $G$ preserving  the  i-th factor of the product. Note that 
the kernel $K_i$ of the induced action $\psi_i \colon G_i \to \Aut(C_i)$  may be non-trivial, as $\psi_i$ is not necessarily injective. 

\begin{rem}\
\begin{enumerate}\label{minaction}
\item
By definition, the groups $G_i$, and therefore also the kernels $K_i$, are permuted under conjugation with elements in $G$. 
Hence, the same holds true for the intersections 
\[
H_i :=K_1\cap \ldots \cap \widehat{K}_i \cap \ldots \cap K_n. 
\]
This observation together with 
$$H_i\cap H_j=\bigcap_{i=1}^n K_i = \lbrace 1 \rbrace \quad \makebox{for all} \quad i \neq j$$ 
implies that  the product 
 $H:=\prod_{i=1}^n H_i$ is normal in  $G$. 
  \item 
 In \cite[Theorem 2.8]{FGR25}, the authors discuss the special case $G=G^0$. They show that an automorphism $f\in \Aut(X)$ lifts to the product $C_1\times \ldots \times C_n$ if all of the subgroups $H_i$ are trivial. Moreover, there is always a unique  realization of $X$ with this property.  It is called the minimal realization of $X$. 
 \item 
 We point out that a minimal realization  always exists, even if $G^0 \subsetneq G$:
For an arbitrary realization $X\simeq  (C_1 \times \ldots \times C_n)/G,$ we obtain the isomorphisms
\[
X\simeq \frac{(C_1\times \ldots\times C_n)}{G} \simeq \frac{(C_1\times \ldots\times C_n)/H}{G/H} \simeq 
\frac{C_1/H_1 \times \ldots \times C_n/H_n}{G/H}. 
\]
By construction, $\widetilde{H}_i = \lbrace 1 \rbrace$ holds for the induced action of $\widetilde{G}:=G/H$ on the product of the curves $\widetilde{C}_i:= C_i/H_i$. 
\end{enumerate}
\end{rem}

\begin{theorem}\label{thmAppendix}

Let $X=(C_1 \times \ldots \times C_n)/G$ be a minimal realization of a varietiy isogenous to a  product of curves, i.e. a realization with trivial  $H_i$. Then every automorphism of $X$ lifts to an automorphism of $C_1 \times \ldots \times C_n$. 
In particular 
\[
\Aut(X) \simeq N_{A}(G)/G,
\]
where $N_{A}(G)$ is the normalizer of $G$ in $A:=\Aut(C_1\times \ldots \times C_n)$. 

\end{theorem}

\begin{proof}
%It suffices to prove that any $f\in \Aut(X)$ lifts to an automorphism of the product $C_1\times \ldots \times C_n$. For this purpose, 
We consider the universal cover $\pi \colon \Delta^n \to X$. 
The fundamental group of $X$ is isomorphic to the group of covering transformations  of $\pi$, which can be written as 
\[
\Gamma:=\lbrace \gamma \in \Aut(\Delta^n) ~\vert ~ \exists ~ g \in G : ~ g\circ \rho = \rho \circ \gamma \rbrace, 
\]
where $\rho\colon \Delta^n \to C_1\times \ldots\times C_n$ is the universal cover of the product of the curves.  
Since  $\Aut(\Delta^n)\simeq \Aut(\Delta)^n \rtimes \mathfrak S_n$, we observe that the intersection 
$\Gamma^0:=\Gamma \cap \Aut(\Delta)^n$ is isomorphic to the fundamental group of the 
associated unmixed quotient 
\[
X^0:=(C_1 \times \ldots \times C_n)/G^0. 
\]
It suffices to  provide a lift of $f$ to an automorphism $\hat{f} \in \Aut(X^0)$, since by minimality we already know that any automorphism of $X^0$ has a  lift to an automorphism of  $C_1\times \ldots \times C_n$, cf. Remark \ref{minaction}. 
For the existence of $\hat{f} \in \Aut(X^0)$, it 
suffices to verify that 
$f_{\ast} \pi_1(X^0)  \subset  \pi_1(X^0) $. On the level of covering transformations, this amounts to the condition  
$$F\circ \Gamma^0 \circ F^{-1} \subset \Gamma^0,$$
where  $F\in \Aut(\Delta^n)$ is a lift of $f$ with respect to the universal cover $\pi$. We now verify this inclusion. 
Clearly, the conjugation $F\circ \Gamma^0 \circ F^{-1}$ is  contained in $\Gamma$, because $F$ is a lift of $f$. On the other hand, it is also contained in $\Aut(\Delta)^n$, because  $\Aut(\Delta)^n$ is 
  normal in $\Aut(\Delta^n)\simeq \Aut(\Delta)^n \rtimes \mathfrak S_n$. Combining both inclusions, we 
obtain $$F\circ \Gamma^0 \circ F^{-1} \subset \Gamma \cap \Aut(\Delta)^n=\Gamma_0$$   
and therfore the desired lift $\hat{f} \in \Aut(X^0)$. 
\end{proof}

\begin{rem}
The argument in the above proof  shows that a minimal realization of a variety isogenous to a product is unique. Indeed any biholomorpism 
\[
f \colon (C_1 \times \ldots \times C_n)/G \to  (D_1 \times \ldots \times D_n)/G'
\]
between two minimal realizations lifts to a biholomorphism $$\widehat{f} \colon C_1 \times \ldots \times C_n \to D_1 \times \ldots \times D_n.$$ Conjugation with $\widehat{f}$ provides an isomorphism between the Galois groups 
$$\widehat{f}\cdot G \cdot \widehat{f}^{-1} =G'.$$
\end{rem}

\vspace{0.8cm}

\noindent Christian Gleissner
\vspace{0.1cm}
\\ University of Bayreuth, Universit\"atsstr. 30, D-95447 Bayreuth, Germany
\\ {\tt Christian.Gleissner@uni-bayreuth.de}

\vspace{0.4cm}

\noindent Noah Ruhland
\vspace{0.1cm}
\\ University of Bayreuth, Universit\"atsstr. 30, D-95447 Bayreuth, Germany
\\ {\tt Noah.Ruhland@uni-bayreuth.de}

\end{document}